\documentclass[12pt,a4paper,twoside,reqno]{amsart}
\title[The Kontsevich--Rosenberg principle for bi-symplectic forms]{The Kontsevich--Rosenberg principle for bi-symplectic forms}

\author[D. Fern\'andez]{David Fern\'andez}
  \address{Instituto de Matem\'atica Pura e Aplicada (IMPA) \and Universidade Federal do Rio de Janeiro \\ 
Estrada Dona Castorina 110
Rio de Janeiro / Brasil 22460-320
    }
 \email{davidfa@impa.br}

\thanks{The author is supported by IMPA
  and CAPES through their postdoctorate of excellence fellowships at
  UFRJ}

\usepackage{hyperref,url}
\usepackage{amssymb,latexsym}
\usepackage{amsmath,amsthm}
\usepackage[latin1]{inputenc}
\usepackage{color}
\usepackage{tikz,times,}
\usepackage{bbm}

\usepackage{enumerate}
\usepackage[all]{xy} \CompileMatrices
\SelectTips{cm}{12} 

\theoremstyle{plain}
\newtheorem{theorem}{Theorem}[section]
\newtheorem{lemma}[theorem]{Lemma}

\newtheorem{proposition}[theorem]{Proposition}

\theoremstyle{definition}
\newtheorem{definition}[theorem]{Definition}
\newtheorem{definition-theorem}[theorem]{Definition-Theorem}
\newtheorem{example}[theorem]{Example}

\theoremstyle{remark}
\newtheorem{remark}[theorem]{Remark}

\newcommand{\secref}[1]{\S\ref{#1}}

\numberwithin{equation}{section} 

\newcommand{\surj}{\to\kern-1.8ex\to}

\newcommand{\lto}{\longrightarrow}
\newcommand{\lra}[1]{\stackrel{#1}{\longrightarrow}}

\newcommand{\defeq}{\mathrel{\mathop:}=} 

\newcommand{\Mat}{\operatorname{Mat}}
\newcommand{\comm}{\operatorname{comm}}
\newcommand{\Aut}{\operatorname{Aut}}
\newcommand{\RR}{\operatorname{R}}
\newcommand{\Proj}{\operatorname{Proj}}

\newcommand{\kk}{{\Bbbk}} 
\newcommand{\op}{{\operatorname{op}}} 
\newcommand{\e}{{\operatorname{e}}} 
\newcommand{\du}{\operatorname{d}\!}
\newcommand{\out}{{\operatorname{out}}} 
\newcommand{\inn}{{\operatorname{inn}}} 

\newcommand{\QQ}{{Q}} 
\newcommand{\Q}[1]{{Q_{#1}}} 

\newcommand{\End}{\operatorname{End}}

\newcommand{\nc}{\operatorname{nc}}
\newcommand{\ab}{\operatorname{ab}}
\newcommand{\Spec}{\operatorname{Spec}}

\newcommand{\GL}{\operatorname{GL}}
\newcommand{\IIm}{\operatorname{Im}}
\newcommand{\Rep}{\operatorname{Rep}}
\newcommand{\DRep}{\operatorname{DRep}}
\newcommand{\Sym}{\operatorname{Sym}}
\newcommand{\T}{\operatorname{T}}
\newcommand{\Hom}{\operatorname{Hom}}
\newcommand{\Der}{\operatorname{Der}}

\newcommand{\DDR}[2]{{\operatorname{DR}_{\nc}^{#1}{#2}}}
\newcommand{\D}{\operatorname{\mathbb{D}er}}

\newcommand{\Tr}{\operatorname{Tr}}
\newcommand{\Id}{\operatorname{Id}}

\definecolor{lemon}{RGB}{189,253,0}

\definecolor{spring}{RGB}{124,252,0}

\definecolor{palegreen}{RGB}{50,205,50}

\definecolor{flojo}{RGB}{238,238,0}

\definecolor{Yellow}{RGB}{254,203,0}

\definecolor{Green}{RGB}{0,133,66}

\definecolor{Blue}{RGB}{0,70,193}

\definecolor{MyDarkBlue}{rgb}{0,0.08,0.45}

\definecolor{greyclaro}{RGB}{118,118,118}

\definecolor{greyoscuro}{RGB}{80,80,80}

\definecolor{greyMUYoscuro}{RGB}{40,40,40}

\definecolor{blue-green}{rgb}{0.0, 0.87, 0.87}

	\definecolor{ao}{rgb}{0.0, 0.0, 1.0}
	
	\definecolor{brandeisblue}{rgb}{0.0, 0.44, 1.0}

\begin{document}
\maketitle

\begin{abstract}
In this expository note, we explain the so-called Van den Bergh functor, which enables the formalization of the Kontsevich--Rosenberg principle, whereby a structure on an associative algebra has geometric meaning if it induces standard geometric structures on its representation spaces. Crawley-Boevey, Etingof and Ginzburg proved that bi-symplectic forms satisfy this principle; this implies that bi-symplectic algebras can be regarded as noncommutative symplectic manifolds. In this note, we use the Van den Bergh functor to give an alternative proof.
\end{abstract}

\section{Introduction}
Noncommutative geometry is a confuse term which is used in different settings with different meanings.
For this reason, we should start by delimiting the meaning of ``noncommutative'' in this note.\\

We will interested in noncommutative algebraic geometry, meaning that our basic object will be a (finitely generated) associative algebra as commutative algebras are the basic objects in familiar algebraic geometry.
Nevertheless, following \cite{Gi05}, we should distinguish two approaches to this fascinating area.
Noncommutative geometry ``in the small'' is devoted to generalize conventional algebraic geometry to the noncommutative realm; typically noncommutative deformations (called \emph{quantizations}) of their commutative counterparts.
On the other hand, noncommutative geometry ``in the large'' is not a generalization of the commutative theory.
In fact, we may mean that it is parallel to the conventional one; whereas the former is governed by the operad of associative algebras (not necessarily commutative) algebras, the latter is ruled by the operad commutative algebras.
This explains why some authors (e.g. \cite{Ta17}) prefer the name ``associative geometry''.
Regarding this note, as it was pointed out in \cite{ACF15}, in a heuristic way, noncommutative algebraic geometry is devoted to the study of associative algebras as if they were algebras of functions on varieties or schemes, i.e., a finitely generated associative algebra is viewed as the algebra of functions on a finite-type ``noncommutative affine scheme''.
The definition of noncommutative (not necessarily affine) spaces is a major and intricate topic which we will not explore here (see \cite{Ro98} and \cite{Ka98} for some interesting approaches).
Sketchy, (see \cite{Smi}, Chapter 3), the idea is that noncommutative spaces are made manifest by the modules that live on them in the same way that the properties of a commutative scheme are manifested by the category of quasi-coherent modules on it.
Now, the modules over a noncommutative space form, by definition, an abelian category, which is the basic object of study in noncommutative geometry.
Hence, the \emph{motto} would be that noncommutative spaces are abelian categories.
So some interesting abelian, triangulated or differential graded categories may be interpreted as noncommutative objects.\\

The aim of this expository note is two-fold. We want to introduce noncommutative algebraic geometry based on the key notion of double derivations on associative algebras, and then formulate the definition of bi-symplectic forms, the noncommutative analogues of symplectic forms.
So, this document may be regarded as a continuation of other expository works as \cite{Gi05} or \cite{Ta17} which dealt with derivations.
Secondly, it is a reflection on the Kontsevich--Rosenberg principle, which establishes a bridge between the noncommutative and commutative settings.
The ultimate goal of this work is to present the \emph{Van den Bergh functor} that realizes and formalizes this principle, and allows us to give a more natural proof of that bi-symplectic forms satisfy the Kontsevich--Rosenberg principle (this was originally proved in \cite{CBEG07}, Theorem 6.4.3 (ii)).
We will follow closely the approach given by Berest, Ramadoss and his coauthors (see \cite{BFR14}). 
Hopefully, this expository note may be used as an introduction to their insightful works.

\subsection*{Contents}
In Section \secref{notation}, we start by introducing some notions and notations that will be used throughout.
Section \secref{repr-quivers} shows that the space of representations of a double quiver is the cotangent bundle of the representation space of the underlying quiver.
Following \cite{ACF17}, the aim of Section \secref{bi-symp} is to define the key notion of bi-symplectic algebras, giving a friendly introduction to noncommutative algebraic geometry based on double derivations.
In the next section, we construct the scheme of representations of a (finitely generated) algebra as the representing object of a functor.
In Section \secref{VdB}, we introduce the Van den Bergh functor, which is representable, allowing us to formalize the Kontsevich--Rosenberg principle, as it will be showed in \secref{KR}.
Then we proved that bi-symplectic forms are the noncommutative analogues of conventional symplectic forms, since they satisfy the Kontsevich--Rosenberg principle.

\subsection*{Disclaimer}
These notes have a expository character and no originality is intended.
The author wishes to give a coherent survey on noncommutative algebraic geometry based on double derivations, as in \cite{CBEG07, VdB08}, which may be useful to post-graduate students or beginners in the area.
As the results are extracted from original articles, we tried to give the exact reference where the reader can learn the proof.
In a sequel of this note, we shall plan to give an introduction to the Kontsevich--Rosenberg principle for double Poisson algebras (in the sense of \cite{VdB08}), using the Van den Bergh functor.

\subsection*{Acknowledgements}
This document was the basis of the talk given by the author in the Symplectic Geometry session of the 31${}^\circ$ Col\'oquio Brasileiro de Matem\'atica, which took place at IMPA, Rio de Janeiro, in August 2017.
I am very grateful to Marta Bator\' eo and Leonardo Macarini, organizers of the session, for the opportunity.
Thanks are due to Alastair King whose questions and suggestions about the original proof of that bi-symplectic forms satisfy the Kontsevich--Rosenberg principle were the seed of this work, 
and to Luis \'Alvarez-C\'onsul, who introduced me to this fascinating area and for its generosity to share with me some of his insights.
Finally, by email, Yuri Berest gave me an elegant proof of Proposition \ref{prop-1-forms}. I am deeply indebted for his help. 
Finally, the author wish to thank Henrique Bursztyn, Alejandro Cabrera, Reimundo Heluani, and Marco Zambon, for useful discussions.
\section{Notation and conventions}
\label{notation}

\subsection{Algebras}
Let $\kk$ be a field of characteristic zero. An \emph{associative algebra} over $\kk$ is a $\kk$-vector space $A$ together with a bilinear map $m\,\colon A\times A\to A$, $(a,b)\mapsto ab$, such that $(ab)c=a(bc)$. A \emph{unit} in an associative algebra $A$ is an element $1\in A$ such that $1a=a1=a$. From now on, by an algebra $A$ we will mean a finitely generated (over $\kk$) associative algebra with a unit. A basic example of an associative algebra is the algebra $\End V$ of endomorphisms of a $\kk$-vector space $V$ to itself (the multiplication is given by the composition). The free algebra $\kk\langle x_1,x_2,...,x_n\rangle$ is an associative algebra, whose basis consists of words in letters $x_1,...,x_n$, and multiplication in this basis is simply the concatenation of words.
A \emph{homomorphism of algebras} $f\colon\, A\to B$ is a linear map such that $f(xy)=f(x)f(y)$ for all $x,y\in A$, and $f(1_A)=1_B$. The category of associative (resp. commutative) $\kk$-algebras will be denoted $\texttt{Alg}_{\kk}$ (resp. $\texttt{CommAlg}_{\kk}$). As usual \texttt{Sets} denotes the category of sets.\\


The unadorned symbols $\otimes=\otimes_\kk$, $\Hom=\Hom_\kk$, will denote the tensor product and the space of linear homomorphisms over the base field.
The opposite algebra and the enveloping algebra of an associative algebra $A$ will be denoted $A^\op$ and $A^\e:=A\otimes A^\op$, respectively. The category of $A$-bimodules  will be denoted $\texttt{Bimod}(A)$. Also, we identify left $A^\e$-modules with $A$-bimodules.\\

The $A$-bimodule $A\otimes A$ has two $A$-bimodule structures, called the \emph{outer} bimodule structure $(A\otimes A)_{\out}$ and the \emph{inner} bimodule structure $(A\otimes A)_{\inn}$, which correspond to the left $A^\e$-module structure ${}_{A^\e}A^\e$ and right $A^\e$-module structure ${}_{({A^\e)}^{\op}}A^\e=(A^\e)_{A^\e}$, respectively. More explicitly,
\begin{align}
a_1(a\otimes b) b_1&= (a_1a)\otimes (bb_1)\quad  \text{ in } (A\otimes A)_\out,
\label{outer}
\\
a_1*(a\otimes b) *b_1&= (ab_1)\otimes (a_1b) \quad \text{ in } (A\otimes A)_\inn.
\label{inner}
\end{align}
Let $M$ be an $A$-bimodule, we define the \emph{bidual} of $M$ as:
\[
M^\vee:=\Hom_{A^\e}(M,(A\otimes A)_\out),
\]
where the $A$-bimodule structure on $M^\vee$ in induced by the one in $(A\otimes A)_\inn$.
An $A$-bimodule $M$ is called \emph{symmetric} if $am=ma$ for every $a\in A$ and $m\in M$. 
Finally, $\Mat_N(\kk)$ denotes the algebra of $N\times N$ matrices with entries in the field $\kk$.

\subsection{Quivers}
In this subsection, we establish some well-known notions and results which enable us to fix notation. We will closely follow the modern references \cite{ARS95} and \cite{ASS06}.\\

A \emph{quiver} $\QQ$ consists of a set $\Q0$ of vertices, a set $\Q1$
of arrows and two maps $t,h\colon\Q1\to\Q0$ assigning to each arrow
$a\in\Q1$, its \emph{tail} and its \emph{head}. We write $a\colon i\to
j$ to indicate that an arrow $a\in\Q1$ has tail $i=t(a)$ and head
$j=h(a)$. Given an integer $\ell\geq 1$, a non-trivial path of length
$\ell$ in $\QQ$ is an ordered sequence of arrows
\[
p=a_\ell\cdots a_1, 
\]
such that $h(a_j)=t(a_{j+1})$ for $1\leq j<\ell$. This path $p$ has tail
$t(p)=t(a_1)$, head $h(p)=h(a_\ell)$, and is represented pictorially as follows.
\[
\bullet\stackrel{a_\ell}{\longleftarrow}\bullet\longleftarrow\cdots\longleftarrow\bullet\stackrel{a_1}{\longleftarrow}\bullet
\]
For each vertex $i\in\Q0$, $e_i$ is the \emph{trivial path} in $\QQ$,
with tail and head $i$, and length $0$. A \emph{path} in $\QQ$ is
either a trivial path or a non-trivial path in $\QQ$. The path algebra
$\kk\QQ$ is the associative algebra with underlying vector space
\[
\kk Q=\bigoplus_{\text{paths $p$}}\kk p,
\]
that is, $\kk Q$ has a basis consisting of all the paths in $Q$, with
the product $pq$ of two non-trivial paths $p$ and $q$ given by the
obvious path concatenation if $t(p)=h(q)$, $pq=0$ otherwise,
$pe_{t(p)}=e_{h(p)}p=p$, $pe_i=e_jp=0$, for non-trivial paths $p$ and
$i,j\in\Q0$ such that $i\neq t(p)$, $j\neq h(p)$, and $e_ie_i=e_i, e_ie_j=0$
for all $i,j\in\Q0$ if $i\neq j$. We will always assume that a
quiver $\QQ$ is finite, i.e. its vertex and arrow sets are finite, so
$\kk Q$ has a unit
\begin{equation}\label{eq:path-algebra-unit}
1=\sum_{i\in\Q0}e_i.
\end{equation}
Define vector spaces
\[
R_\QQ=\bigoplus_{i\in \Q0}\kk e_i,\quad V_\QQ=\bigoplus_{a\in \Q1}\kk a.
\]
Then $R_\QQ\subset \kk\QQ$ is a semisimple commutative (associative)
algebra, 
 because it is the subalgebra spanned by the trivial
paths, which are a complete set of orthogonal idempotents of $\kk\QQ$,
i.e. $e_i^2=e_i$, $e_ie_j=0$ for $i\neq j$, and, by \eqref{eq:path-algebra-unit},
\[
x=\sum_{i,j\in\Q0}e_jxe_i, \text{ for all $x\in \kk\QQ$}
\]
Furthermore, as $V_\QQ$ is a
vector space with basis consisting of the arrows, it is an
$R_\QQ$-bimodule with multiplication $e_jae_i=a$ if $a\colon i\to j$
and $e_iae_j=0$ otherwise, and the path algebra is the tensor algebra
of the bimodule $V_\QQ$ over $R\defeq R_\QQ$, that
is (see Proposition 1.3 in \cite{ARS95}),
\begin{equation}\label{eq:grPathAlg.1}
\kk\QQ=\T_{R}V_\QQ,
\end{equation}
where a path $p=a_\ell\cdots a_1\in k\QQ$ is identified with a tensor
product $a_\ell\otimes\cdots\otimes a_1\in \T_{R}V_\QQ$.\\

Let $A=\kk\QQ$. It is well known
\footnote{ See e.g. \S
 4.6 of \url{https://www.math.uni-bielefeld.de/~sek/kau/text.html}.}
that the decomposition
\[
A=\bigoplus_{i\in\Q0}Ae_i,
\]
is a decomposition of the $A$-module ${}_AA$ as a direct sum of
pairwise non-isomorphic indecomposable projective $A$-modules. Note
that the vector space underlying $Ae_i$ has a basis consisting of all
the paths in $\QQ$ with tail $i$. In fact, $\{ Ae_i\mid i\in\Q0\}$ is
a complete
 set of indecomposable finitely generated $A$-modules up to isomorphism (see, for instance, \cite{ASS06}). Furthermore, the evaluation map
\begin{equation}\label{eq:indecomp-proj}
\Hom_A(Ae_i,M)\lra{\cong} e_iM\,\colon \quad f\longmapsto f(e_i)
\end{equation}
is a natural isomorphism, for all $i\in\Q0$ and all $A$-modules $M$.\\

\section{Representations of double quivers}
\label{repr-quivers}

\subsection{Representations of quivers}
Let $Q=(Q_0,Q_1, h, t)$ be a fixed quiver whose path algebra will be denoted $\kk Q$.
A \emph{representation of a quiver} $Q$ is the following collection of data
\begin{enumerate}
\item [\textup{(i)}]
For every vertex $i\in Q_0$, a $\kk$-vector space $V_i$;
\item [\textup{(ii)}]
For every arrow $a\in Q_1$, $a\,\colon i\to j$, a linear operator $X_a\,\colon V_i\to V_j$.
\end{enumerate}
We denote the representation by $V$.
A morphism of representations $f\,\colon V\to W$ is a collection of linear operators $f_i\,\colon V_i\to W_i$ which commute with the operators $X_a$: if $a\,\colon i\to j$, then $f_j\circ X_a=X_a\circ f_i$; in other words
\[
\xymatrix{
V_i\ar[d]_-{X_a}\ar[r]^-{f_i} & W_i\ar[d]^-{X_a}
\\
V_j\ar[r]^-{f_j} & W_j
}
\]
Morphisms $V\to W$ form a vector space, which we shall denote by $\Hom(V,W)$. Also, $\End V $ stands for the \emph{algebra of endomorphisms} of a representation $V$, and $\Aut V=\{\phi\in\End V\mid f\text{ is invertible}\}$ will be the \emph{group of automorphisms} of $V$. Throughout, we will only consider finite-dimensional representations; i.e. those where each space $V_i$ is finite-dimensional, whose dimension will be denoted $v_i$.

\begin{example}
Let $Q$ be the \emph{Jordan quiver}, that is, the quiver with one vertex and one arrow. Then a representation of this quiver is a pair $(V,X)$, where $V$ is a $\kk$-vector space, and $X\,\colon V\to V$ is a linear map. Hence, classifying representations of $Q$ is equivalent to classifying linear operators up to change of basis, or matrices up to conjugacy.
 If $\kk$ is an algebraically closed field (i.e. $\kk=\overline{\kk}$), the classification is given by the Jordan canonical form; if $\kk$ is not algebraically closed, the answer is more difficult.
\end{example}

It is clear that finite-dimensional representations of the quiver $Q$ form a category, denoted by $\Rep Q$. In fact, this category is endowed with direct sums, subrepresentations, quotients, kernels and images (similar to those in the category of group representations). Remarkably, images and kernels satisfy the usual properties such as $\IIm f\simeq V/(\ker f)$; thus $\Rep Q$ is an abelian category over $\kk$.  This also follows from the equivalence of $\Rep Q$ with the category of modules over the path algebra $\kk Q$ (see, for instance, \cite{Smi}, Theorem 5.9).\\

Given a finite-dimensional representation of the quiver $Q$, denote $\textbf{v}=\dim V=\sum_{i\in Q_0} v_i$. Choosing a basis in each $V_i$, we can identify $V_i\simeq\kk^{v_i}$. The structure of a representation of $Q$ is described by a collection of matrices $X_a\in\Hom(\kk^{v_i},\kk^{v_j})=\Mat_{v_j\times v_i}(\kk)$, for each edge $a\,\colon i\to j$, or equivalently, by a vector $X$ in the space
\begin{equation}
\RR(Q,\textbf{v})=\bigoplus_{a\in Q_1}\Hom(\kk^{v_{t(a)}},\kk^{v_{h(a)}}).
\label{def-repr-inicial}
\end{equation}
Conversely, every $X\in\RR(Q,\textbf{v})$ defines a representation $V^{X}=(\{\kk^{v_i}\},\{X_a\})$. We shall occasionally use the notation $\RR(V)=\bigoplus_{a\in Q_1}\Hom_\kk(V_{t(a)},V_{h(a)})$.

\subsection{Representations of double quivers}
For any pair of finite-dimensional vector spaces $V$, $W$ we have a canonical pairing $\Hom(V,W)\otimes\Hom(W,V)\to\kk$, given by $(f,g)\mapsto \Tr(fg)$. By the properties of the trace, this pairing is symmetric and nondegenerate; consequently, it defines an isomorphism 
\begin{equation}
\Hom(W,V)\simeq \Hom(V,W)^*.
\label{ism-vector-hom}
\end{equation}

Now, one of the most important class of examples of symplectic manifolds are the cotangent bundles. Let $X$ be a manifold of dimension $n$, and let $T^*X$ be its cotangent bundle
\[
T^*X=\{(x,\lambda)\}_{x\in X,\; \lambda\in T^*_x X},
\]
which is a manifold of dimension $2n$. Then $T^*X$ is endowed with a canonical 1-form $\alpha$ (called the \emph{Liouville form}). In fact, $T^*X$ has a canonical symplectic structure given by the 2-form $\omega=\du\alpha$. Explicitly, if $\{q^i\}$ are local coordinates on $X$ and $\{p_i\}$ are corresponding coordinates on $T^*_xX$, so that the $2n$-tuple $\{p_i,q^i\}$ are local coordinates on $T^*X$, then locally the symplectic form $\omega$ is given by
\[
\omega=\sum_i\du p_i\wedge\du q^i.
\]

In the case of a finite-dimensional vector space $E$, we have $T^*E=E\oplus E^*$, and the symplectic form on $T^*E$ is given by
\begin{equation}
\omega\left((v_1,\lambda_1),(v_2,\lambda_2)\right)=\langle \lambda_1,v_2\rangle-\langle \lambda_2,v_1\rangle.
\label{symplectic-form-vectro space}
\end{equation}
Here $\langle-,-\rangle$ stands for the canonical pairing between $E$ and its dual.\\

Let $\overline{Q}$ be the double quiver obtained from $Q$ by adding an additional arrow $a^*\,\colon j\to i$ for every arrow $a\,\colon i\to j$ in $Q$. By \eqref{def-repr-inicial} and \eqref{ism-vector-hom}, it is immediate that
\begin{align*}
\RR(\overline{Q},\textbf{v})&=\bigoplus_{a\in \overline{Q}_1}\Hom(\kk^{v_{h(a)}},\kk^{v_{t(a)}})
\\
&=\bigoplus_{a\in Q_1}\Hom(\kk^{v_{h(a)}},\kk^{v_{t(a)}})\oplus \bigoplus_{a\in \overline{Q}_1\setminus Q_1}\Hom(\kk^{v_{t(a)}},\kk^{v_{h(a)}})
\\
&=\RR(Q,\textbf{v})\oplus \RR(Q,\textbf{v})^*
\\
&=T^*\RR(Q,\textbf{v}).
\end{align*}

\section{Bi-symplectic algebras}
\label{bi-symp}

\subsection{Noncommutative differential forms}
Given an $A$-bimodule $M$, a \emph{derivation} of $A$ into $M$ is an
additive map $\theta\colon A\to M$ satisfying the Leibniz rule
$\theta(ab) =(\theta a)b+a(\theta b)$ for all $a,b\in A$.
The space $\Der(A,M)$ of
$\kk$-linear derivations of $A$ into $M$ is a $Z(A^\e)$-module.

If we denote by $m\colon A\otimes A\to A$ the multiplication map of $A$, we take $\Omega^1_{\nc} A$ as the kernel of $m$ (seen as a sub-$A$-bimodule of $A\otimes A$), and the derivation 
\begin{equation}
\du\,\colon A\lto \Omega^1_{\nc} A\colon \quad a\longmapsto   \du a=1\otimes a-a\otimes 1
\label{univ-derivation}
\end{equation}
is called the \emph{universal derivation}. Similarly to the commutative case, the pair $(\Omega^1_{\nc}A,\du\,)$ satisfies the following universal property:

\begin{lemma}[\cite{CQ95}, \S 2]
There exists a unique pair
$(\Omega^1_{\nc}A,\du\,)$ (up to isomorphism), where $\Omega^1_{\nc} A$ is an
$A$-bimodule and $\du\,\colon A\to\Omega^1_{\nc} A$ is a $\kk$-linear
derivation $\theta\colon A\to M$, satisfying the following universal
property: for all pairs $(M,\theta)$ consisting of an $A$-bimodule $M$
and a $\kk$-linear derivation, there exists a unique $A$-bimodule
morphism $i_\theta\colon\Omega^1_{\nc} A\to M$ such that
$\theta=i_\theta\circ\du\,$. 
\label{Univ-Property-Kaehler-diff}
\end{lemma}

If $a,b\in A$, it will be useful to note that $\Omega^1_{\nc}A$ is generated as an $A$-bimodule by the symbols $\du a$ subject to the usual relations: $\du\,(ab)=a(\du b)+(\du a)b$ and linearity.\\


The algebra of \emph{non-commutative differential forms} of $A$ is the tensor algebra
\begin{equation}
\Omega^\bullet_{\nc} A:=T_A\Omega^1_{\nc} A
\label{NC-diff-forms}
\end{equation}
of the $A$-bimodule $\Omega^1_{\nc} A$ if $n\geq 0$, and $\Omega^n_R A=0$ if $n<0$; by convention $\Omega^0_{\nc}A=A$. 
Also, every homogeneous element $\beta$ in $\Omega^\bullet_{\nc}A$ has a representation $a\du a_1\cdots\du a_n$ ($a,a_1,...,a_n\in A$).
Extending \eqref{univ-derivation} by the Leibniz rule, $\Omega^\bullet_{\nc} A$ becomes a differential graded algebra (DG-algebra for short).
There exists a natural inclusion map $i_{\nc}\,\colon A\hookrightarrow\Omega^\bullet_{\nc}A$, and the DG-$A$-algebra $\Omega^\bullet_{\nc}A$ satisfies the following universal property:
\begin{lemma}[\cite{CQ95}]
Let $A$ be a $\kk$-algebra. For any DG-algebra $C$ and algebra morphism $g\,\colon A\to C^0$ into the zero-degree component of $C$, there exists a unique morphism of DG-algebras $h\,\colon\Omega^{\bullet}_{\nc}A\to C$ making the diagram commute:
\[
\xymatrix{
A\ar[dr]_-g \ar@{^{(}->}[r]^-{i_{\nc}} & \Omega^\bullet_{\nc}A\ar@{.>}[d]^-h
\\
& C
}
\]
\label{dg-univ-property}
\end{lemma}

A key technical point in non-commutative algebraic geometry is that $(\Omega^\bullet_{\nc} A,\du\,)$ has trivial cohomology (see \cite{Gi05}, \S 11.4). To obtain a more interesting theory, we define the \emph{non-commutative Karoubi--de Rham complex} of $A$ as the graded vector space
\begin{equation}
\DDR{\bullet} {A}=\Omega^\bullet_{\nc} A/[\Omega^\bullet_{\nc} A,\Omega^\bullet_{\nc} A],
\label{Karoubi-de-Rham}
\end{equation}
where $[-,-]$ denotes the graded commutator.
Note that the differential $\du\,\colon \Omega^\bullet_{\nc} A\to \Omega^{\bullet+1}_{\nc} A$ descends to a well-defined differential $\du\,\colon \DDR{\bullet}{A}\to \DDR{\bullet+1}{A}$; so $\DDR{\bullet}{A}$ becomes a differential graded vector space.\\

\subsection{Double derivations}
 To study associative algebras as if they were algebras of functions on a noncommutative space, we will use an analogue of vector fields in this context. It is well known that a regular vector field on a smooth affine algebraic variety $X$ is equivalent to a derivation $\kk[X] \to \kk[X]$ of the coordinate ring of $X$, i.e., derivations of a commutative algebra $C$ play the role of vector fields. It has been commonly accepted until recently that this point of view applies to noncommutative algebras $A$ as well. Nevertheless, Crawley-Boevey \cite{CB99} showed that for a smooth affine curve $X$ with coordinate ring $A := \kk[X]$, the algebra of differential operators on $X$ can be constructed by means of \emph{double derivations}. To define them, it will be crucial to reformulate the universal property expressed by Proposition \ref{Univ-Property-Kaehler-diff} by saying that the $A$-bimodule
$\Omega^1_{\nc}A$ represents the functor $\Der(A,-)$ from the category
of $A$-bimodules into the category of $\kk$-modules. Hence there
exists a canonical isomorphism of $A$-bimodules
\begin{equation}\label{sub:ncdiffforms.corep}
\Der(A,M)\lra{\cong}\Hom_{A^\e}(\Omega^1_{\nc}A,M)\colon\quad \theta\longmapsto i_\theta,
\end{equation}
whose inverse map is given by
$i_\theta\longmapsto\theta=i_\theta\circ\du\,$. In particular, since
$i_\theta$ is an $A^\e$-module morphism, $i_\theta(a\du b)=a
\theta (b)$, for all $a\in A^\e, b\in A$.\\

Dualizing the $A$-bimodule $\Omega^1_{\nc}A$, we obtain another $A$-bimodule
\begin{equation}
(\Omega^1_{\nc}A)^\vee=\Hom_{A^\e}(\Omega^1_{\nc}A,{}_{A^\e}A^\e)=\Hom_{A^\e}(\Omega^1_{\nc}A,(A\otimes A)_\out),
\label{double canonical isomorphism}
\end{equation}
where the $A$-bimodule structure
comes from the inner $A$-bimodule
structure. By the universal
property~\eqref{sub:ncdiffforms.corep} of $\Omega^1_{\nc}A$, we have a
canonical isomorphism
\begin{equation}\label{eq:double-dual-1}
\D{A}\lra{\cong}(\Omega^1_{\nc}A)^\vee, \quad
\Theta\longmapsto i_\Theta,
\end{equation}
where
\begin{equation}
\D{A}\defeq\D(A,{}_{A^\e}A^\e)=\Der_R(A,(A\otimes A)_\out)
\label{def-doble-deriv}
\end{equation}
is the $A$-bimodule of \emph{double derivations}, whose $A$-bimodule
structure also comes from the inner $A$-bimodule structure.




\subsection{Smoothness}
\label{smoothness}
Note that if the associative algebra $A$ is finitely generated over $\kk$, then
$\Omega^1_{\nc}A$ is a finitely generated
$A^\e$-module. So, the associative $\kk$-algebra $A$ is called \emph{smooth} over $\kk$ if it
is finitely generated over $\kk$ and the $A^\e$-module $\Omega^1_{\nc}A$ is projective. 



\subsection{The (reduced) contraction operator}
Fix $\Theta\in\D A$. From \eqref{sub:ncdiffforms.corep}, taking $M=A\otimes A$, for any 1-form $\alpha\in\Omega^1_{\nc} A$, we have the following $A$-bimodule map, called the \emph{contraction operator}
\begin{equation}
i_\Theta\colon\Omega^1_{\nc} A\lto A\otimes A\colon\quad \alpha\longmapsto i_\Theta\alpha=i^\prime_\Theta\alpha\otimes i^{\prime\prime}_\Theta\alpha.
\label{contraction-operator-double-deriv}
\end{equation}

\begin{remark}
From now on, we will systematically use symbolic Sweedler's notation and we will omit the summation sign for an element in the tensor product. Similarly, we write the map $\Theta\colon A\to A\otimes A$ as $a\mapsto \Theta^\prime(a)\otimes \Theta^{\prime\prime}(a)$.
\end{remark}

Note that on generators, $i_\Theta$ acts as $i_\Theta(a)=0$, and $i_\Theta(\du b)=\Theta(b)$ for all $ a\in A $, $\du b\in\Omega^1_{\nc}A$. 
Next, since $\Omega^{\bullet}_{\nc}A$ in \eqref{NC-diff-forms} is the free algebra of the graded bimodule $\Omega^1_{\nc}A$, there exists a unique extension of the map $i_{\Theta}\,\colon\Omega^1_{\nc}A\to A\otimes A$ to a double derivation of degree -1 on $\T_A(\Omega^1_{\nc}A)$, 
\begin{equation}
i_{\Theta} \; \colon \Omega^{\bullet}_{\nc} A\lto\bigoplus (\Omega^{i}_{\nc} A \otimes \Omega^{j}_{\nc} A),
\label{extension-contraccion}
\end{equation}
 where the sum is over pairs $(i,j)$ with $i+j=\bullet-1$. Note that we regard $\Omega^{\bullet}_{\nc}A\otimes\Omega^{\bullet}_{\nc}A$ as an $\Omega^{\bullet}_{\nc}A$-bimodule with respect to the outer bimodule structure (see \eqref{outer}). 
 Moreover, sometimes, we will view the contraction map $i_\Theta$ as a map $\Omega^{\bullet}_{\nc}A\to (T_A(\Omega^\bullet_{\nc}A))^{\otimes 2}$. 
 Explicitly (see \cite{CBEG07} (2.6.2)), for any $n=1,2,...$, and $\alpha_1,...,\alpha_n\in\Omega^1_{\nc} A$, we write:
 \begin{equation}
 i_\Theta(\alpha_1\alpha_2\cdots\alpha_n)=\sum_{1\leq k\leq n}(-1)^{k-1}(\alpha_1\cdots\alpha_{k-1}(i^\prime_\Theta\alpha_k))\otimes((i^{\prime\prime}_\Theta\alpha_k)\alpha_{k+1}\cdots\alpha_n).
 \label{formula-larga-contraccion}
 \end{equation}
 
 Now, given a graded $\kk$-algebra $C$ and $c=c_1\otimes c_2$, with $c_1,c_2\in C$, we define ${}^\circ c\defeq(-1)^{\lvert c_1\rvert\lvert c_2\rvert} c_2 c_1$, and given a linear map $\phi\colon C\to  C^{\otimes 2}$, write ${}^\circ \phi\,\colon C\to C\colon$ $  c\mapsto {}^\circ (\phi(c))$. 
In our case, we set $C=\Omega^\bullet_{\nc} A$ and we define the \emph{reduced contraction operator}
 \begin{equation}
\iota_\Theta \colon \Omega^\bullet_{\nc} A \lto \Omega^\bullet_{\nc} A\colon\quad 
 \alpha \longmapsto {}^\circ(i_\Theta)= (-1)^{\lvert i^{\prime}_\Theta(\alpha)\rvert\lvert i^{\prime\prime}_\Theta(\alpha)\rvert}i^{\prime\prime}_\Theta (\alpha) i^\prime_\Theta(\alpha).
\label{circulo derecha}
\end{equation}
Explicitly, for any $\alpha_1,\alpha_2,...,\alpha_n\in\Omega^1_R A$, using the definition of $i_\Theta$, we have
\begin{equation}
\iota_\Theta(\alpha_1\cdots\alpha_n)=\sum^n_{k=1}(-1)^{(k-1)(n-k+1)}(i^{\prime\prime}_\Theta\alpha_k)\alpha_{k+1}\cdots\alpha_n\alpha_1\cdots\alpha_{k-1}(i^\prime_\Theta\alpha_k).
\label{reduced-contraction-general}
\end{equation}

\subsection{Bi-symplectic algebras}
\begin{definition}[\cite{CBEG07}]
Let $A$ be an associative $\kk$-algebra. An element $\omega\in 
\DDR{2}(A)$ which is closed for the universal derivation $\du$ is a
\emph{bi-symplectic form} if the following map of $A$-bimodules is an isomorphism: 
\[
\iota(\omega)\colon\D A\lra{\cong} \Omega^{1}_{\nc}A\colon \quad
\Theta \longmapsto \iota_{\Theta}\omega.
\]
\label{bi-sympldef}
\end{definition}

\section{The scheme of representations}
\label{scheme-repr}

\subsection{A first description of $\Rep(A,V)$}
\label{sub:first-repr}
Very roughly speaking, representation theory deals with symmetry in linear spaces. So, it is not surprising that it may be applied in branches as geometry, probability, quantum mechanics or quantum field theory. In algebra, a central topic is the study of representations of associative algebras.\\

From now on, we fix a $\kk$-vector space $V$ of finite dimension $N$. A \emph{representation} of $A$ consists of a vector space $V$ together with a homomorphism of algebras $\rho\,\colon A\to \End V$. The associated representation space of $A$ parametrizing its representation is defined as
\[
\Rep(A,V)=\Hom(A,\End V).
\]
It is easy to see that $\Rep(A,V)$ is an affine variety and its coordinate ring 
\[
A_V:=\kk[\Rep(A,V)]
\]
may be conveniently described since it is generated by symbols $(a_{ij})_{i,j=1,...,N}$ for all $a\in A$, subject to the relations (see \cite{VdB08a})
\[
(\lambda a)_{ij}=\lambda a_{ij},\quad 1_{jl}a_{j^\prime l^\prime}=\delta_{l j^\prime}a_{j^\prime l^\prime}, \quad a_{jl}+b_{jl}=(a+b)_{jl}, \quad (ab)_{ij}=a_{il}b_{lj},
\]
where $\lambda\in\kk$ and we sum over repeated indices.

\begin{example}
Let $F=\kk\langle x_1,...,x_d\rangle$ be the free associative algebra in $d$ generators. Then any $N$-dimensional representation of $A$ is determined by declaring that the $N\times N$-matrix $X_i$ is the image of the generator $x_i$ for every $i=1,...,d$. Then
\[
\Rep(F,V)=\Mat_{N}(\kk)\oplus\cdots\oplus\Mat_N(\kk)=\Mat_N(\kk)^{\oplus d}.
\]
Then $\kk[\Rep(F,V)]$ is isomorphic to the polynomial ring over $\kk$ generated by the $dN^2$ indeterminates $\{x_{i,jk}\}_{i=1,...,d,\, j,k=1,...,N}$ representing each entry in a generic $d$-tupla of $N\times N$-matrices:
\[
X_1=\begin{pmatrix}
x_{1,11} & \cdots & x_{1,1N}
\\
\vdots & \ddots & \vdots
\\
x_{1,N1} & \cdots & x_{1,NN}
\end{pmatrix},
\quad \cdots,\quad 
X_d=\begin{pmatrix}
x_{d,11} & \cdots & x_{d,1N}
\\
\vdots & \ddots & \vdots
\\
x_{d,N1} & \cdots & x_{d,NN}
\end{pmatrix}.
\]
Furthermore,  if $A$ is finitely generated, there exists a natural number $d\in\mathbb{N}$ such that $A$ may be presented as a quotient
\[
A\simeq F/I
\]
of $F$ by a two-sided ideal $I$. Then an $N$-dimensional representation of $A$ can be specified by a $d$-tuple of $N\times N$ matrices $(X_1,...,X_d)$ such that the map $\kk\langle x_1,...,x_d\rangle\to\Mat_N(\kk)$ defined by $x_i\mapsto X_i$ descends to the quotient, which is true if and only if the matrices $(X_1,...,X_d)$ satisfy every relation determined by the ideal $I$.
\end{example}

From this example it is easy to see that if $A$ is finitely generated, then $\kk[\Rep(A,V)]$ is finitely generated as well.\\

Nevertheless, this perspective on $\kk[\Rep(A,V)]$ despite of being very explicit, it is very uneconomical and it hides the rich algebraic and geometric structure carried by $\Rep(A,V)$.

\subsection{$\Rep( A,V)$ as a representing object}
Using ideas of noncommutative algebra \cite{Be74, Co79, LBW02}, the representation space $\Rep(A,V)$ can be defined in terms of a functor on the category of commutative algebras
\begin{equation}
\Rep_V A\,\colon\texttt{CommAlg}_\kk\lto\texttt{Sets}\colon\quad C\mapsto\Hom_{\texttt{Alg}_\kk}(A,\End V\otimes C).
\label{Rep-functor-comm}
\end{equation}
Following \cite{BKR13}, to prove its representability, the idea is to extend \eqref{Rep-functor-comm} from $\texttt{CommAlg}_\kk$ to the category of all associative $\kk$-algebras:
\begin{equation}
\xymatrix{
\texttt{CommAlg}_\kk \ar[rr]^-{\Rep_V A} \ar@{^{(}->}[d]& & \texttt{Sets}
\\
\texttt{Alg}_\kk\ar@{.>}[urr]_-{\widetilde{\Rep}_V A}
}
\label{diagrama-1}
\end{equation}
The functor $\widetilde{\Rep}_V A$ is defined by the same formula as $\Rep_V A$ in \eqref{Rep-functor-comm}, but the commutative algebra $C$ is replaced by an associative algebra $B$:
\begin{equation}
\widetilde{\Rep}_V A\,\colon\texttt{Alg}_\kk\lto\texttt{Sets}\colon\quad B\mapsto\Hom_{\texttt{Alg}_\kk}(A,\End V\otimes B).
\label{Rep-functor-assoc}
\end{equation}
This functor is representable since its representing object has a very explicit algebraic presentation. Let $A*\End V$ be the free product of $A$ and $\End V$ as $\kk$-algebras (i.e. the coproduct in the category $\texttt{Alg}_k$), and $(-)^{\End V}$ denotes the centralizer of the image of $\End V$ in $A$:
\[
(A*\End V)^{\End V}:=\{v\in A*\End V\mid [v,\varphi]=0, \text{ for all } \varphi\in \End V\}.
\]
Then we define the functor
\begin{equation}
\sqrt[V]{-}\,\colon\texttt{Alg}_\kk\lto \texttt{Alg}_\kk\,\colon A\mapsto (\End V*A)^{\End V}.
\label{functor-sqrt}
\end{equation}


In \cite{Co79}, $\sqrt[V]{A}$ is thought of as the coordinate ring of a noncommutative affine scheme. Indeed, 

\begin{lemma}[\cite{BKR13}, Lemma 2.1]
Let $A, B\in\texttt{Alg}_\kk$. The natural functor $\texttt{Alg}_\kk\to\texttt{Alg}_{\End V}$, $B\mapsto \End V\otimes B$ is an equivalence the categories, where $\End V\otimes B$ is regarded as an object in $\texttt{Alg}_\kk$ using the canonical map $\End V\to \End V\otimes B$.
\end{lemma}

This is the key ingredient to prove the following important result:

\begin{proposition}[\cite{Co79}, \cite{BKR13}, Proposition 2.1(a)]
Let $A, B\in\texttt{Alg}_\kk$
\begin{enumerate}
\item [\textup{(i)}]
There exists a natural bijection
\[
\Hom_{\texttt{Alg}_\kk}(\sqrt[V]{A},B)\simeq \Hom_{\texttt{Alg}_\kk}(A,\End V\otimes B);
\]
\item [\textup{(ii)}]
The functor \eqref{functor-sqrt} is representable, with representing object $\sqrt[V]{A}$.
\end{enumerate}
\label{prop-repr-sqrt}
\end{proposition}

Then we define the functor
\begin{equation}
(-)_V\colon\,\texttt{Alg}_\kk\lto \texttt{CommAlg}_\kk\colon\quad A\longmapsto (\sqrt[V]{A})_{\ab}
\label{functor-V}
\end{equation}
where $(-)_{\ab}$ is the \emph{abelianization functor}, the left adjoint functor to the inclusion functor $\texttt{CommAlg}_\kk\hookrightarrow \texttt{Alg}_\kk$ in \eqref{diagrama-1}, defined as $C\longmapsto C/[C,C]$, where $[C,C]$ denotes the two-sided ideal generated by $[C,C]$, and $[-,-]$ stands for the commutator. The following important result is immediate from Proposition \ref{prop-repr-sqrt}

\begin{proposition}
Let $A\in \texttt{Alg}_\kk$ and $B\in \texttt{CommAlg}_\kk$. Then
\begin{enumerate}
\item [\textup{(i)}]
There exists a natural bijection
\begin{equation}
\Hom_{\texttt{CommAlg}_\kk}(A_V,B)\simeq \Hom_{\texttt{Alg}_\kk}(A,\End V\otimes B);
\label{adjunction-1}
\end{equation}
\item [\textup{(ii)}]
The commutative algebra $A_V$ represents the functor \eqref{Rep-functor-comm}.
\end{enumerate}
\label{prop-adjunction-repr-sch}
\end{proposition}

If we take $B=A_V$ in Proposition \ref{prop-adjunction-repr-sch}(i), we can consider the identity on the left-hand side, $\Id_{A_V}\,\colon A_V\to A_V$.
Then we define the \emph{universal representation} 
\begin{equation}
 \pi\,\colon A\lto\End V\otimes A_V
 \label{pi}
\end{equation}
in $\Hom_{\texttt{Alg}_\kk}(A,\End V\otimes A_V)$ that corresponds to $\Id_{A_V}$ under the adjunction \eqref{adjunction-1}.
It is universal, which enables us to define the functor $(-)_V$ on morphisms (see \cite{Kh12}, Corollary 6).
Let $f\,\colon A_1\to A_2$ be a morphism of associative algebras, and we consider the diagram
\[
  \xymatrix{
  A_1\ar[r]^-{\pi_1} \ar[d]^-f & \End V\otimes (A_1)_V
  \\
    A_2\ar[r]^-{\pi_2} & \End V\otimes (A_2)_V
    }
  \]
  Applying the universal property of $\pi_1$ to the morphism $\pi_2\circ f\,\colon A_1\to \End V\otimes (A_2)_V$, we obtain a unique morphism
  \[
   g\,\colon (A_1)_V\lto (A_2)_V
  \]
such that $\Id_{\End V}\otimes g$ makes the diagram commute. 
We set $(f)_V:=g$.

\subsection{The $\GL(V)$-action on $\Rep(A,V)$}
Let $\GL(V)\subset\End V$ be the group of invertible endomorphisms of $V$. 
The natural left action by conjugation of $\GL(V)$ on $\End V$ induces a left action on $\End V\otimes A_V $ given by
$g\cdot (\varphi\otimes x)=g\varphi g^{-1}\otimes x$, for all $g\in\GL(V)$, $\varphi\in\End V$, and $x\in A_V$.\\

Representation schemes have become essential in representation theory because they enable the use of geometric methods in the study of the representation theory of the algebra $A$, which constitutes an additional instance of the fruitful geometric interaction between algebra and geometry. \\

The naif idea is to consider the quotient of $\Rep(A,V)$ by $\GL(V)$ and then study its orbit space.
However, this \emph{topological quotient} is badly behaved in most cases; it does not carry a reasonable Hausdorff topology.\\

To overcome this problem, we define a \emph{categorical quotient}
\[
 \Rep(A,V)//\GL(V)=\Spec\left(\kk[\Rep(A,V)]^{\GL(V)}\right),
\]
where $\kk[\Rep(A,V)]^{\GL(V)}$ is the algebra of $\GL(V)$-invariant polynomial functions on $\Rep(A,V)$, and $\Spec(-)$ is the set of its maximal ideals.
Using a theorem due to Hilbert, we can prove that this algebra is finitely generated, so $ \Rep(A,V)//\GL(V)$ is an affine algebraic variety.
Since every $G$-orbit $\mathbb{O}\subset \Rep(A,V)$ defines a maximal ideal $J_{\mathbb{O}}\subset\{f\mid f\vert_{\mathbb{O}}\}\subset\kk[\Rep(A,V)]^{\GL(V)}$;
 we have a natural surjective map $\Rep(A,V)/\GL(V)\to \Rep(A,V)//\GL(V)$.
In fact, there exists an isomorphism as topological spaces (\cite{Kir16}, Theorem 9.5): $\Rep(A,V)//\GL(V)\leftrightarrow \{\text{closed orbits in }\Rep(A,V)\}\colon$ $[x]\mapsto \text{ unique closed orbit contained in }\overline{\mathbb{O}}_x$.\\

The problem is that in categorical quotients a lot of geometric information may be lost.
For instance, given a quiver $Q$ without oriented cycles, $\kk[\Rep(\kk Q,V)]^{\GL(V)}=\kk$, hence, we have $\Rep(\kk Q,V)//\GL(V)=pt$.\\

To avoid this dramatic loss of information, Mumford developed \emph{Geometric Invariant Theory} (see \cite{MFK94,Th06}), which consists of a general theory of quotients by a reductive group action via stability conditions, which will require a switch from affine to projective varieties.
Let $\chi$ be a character of $\GL(V)$, that is, a morphism of algebraic groups $\chi\,\colon\GL(V)\to\kk^\times$, and define 
\[
 \kk[\Rep(A,V)]^{\GL(V),\chi}=\{f\in\kk[\Rep(A,V)]\mid f(g\cdot \rho)=\chi(g)f(\rho)\}
\]
It is immediate from the definition that
\[
 A_\chi=\bigoplus_{n\geq 0}\kk[\Rep(A,V)]^{\GL(V),\chi^n}
\]
is a graded algebra; in fact is finitely-generated. 
Thus, we can define the corresponding quasi-projective variety:
\[
 \Rep(A,V)//_\chi\GL(V)=\Proj(A_\chi)=\Proj\left(\bigoplus_{n\geq 0}\kk[\Rep(A,V)]^{\GL(V),\chi^n}\right).
\]
In general, for $n=0$, we have $\kk[\Rep(A,V)]^{\GL(V),\chi^n}=\kk[\Rep(A,V)]^{\GL(V)}$, is the free algebra of $\GL(V)$-invariants.
Thus, we have a canonical algebra embedding $\kk[\Rep(A,V)]^{\GL(V)}\hookrightarrow A_\chi$ as the degree zero subalgebra.
It is well-known that this embedding induces a projective morphism of varieties $\pi\,\colon \Rep(A,V)//_\chi \GL(V)\to \Rep(A,V)//\GL(V)$.\\

Given a nonzero homogeneous semi-invariant $f\in A_\chi$, we take $(\Rep(A,V))_f:=\{x\in \Rep(A,V)\mid f(x)\neq 0\}$.
In GIT, one is interested in some distinguished subsets:

\begin{definition}
 \begin{enumerate}
  \item [\textup{(i)}]
  A point $x\in\Rep(A,V)$ is called $\chi$\emph{-semistable} if there exists $n\geq 1$ and a $\chi^n$-seminvariant $f\in\kk[\Rep(A,V)]^{\chi^n}$ such that $x\in X_f$.
    \item [\textup{(ii)}]
    A point $x\in\Rep(A,V)$ is called \emph{$\chi$-stable} if there exists $n\geq 1$ and a $\chi^n$-semi-invariant $f\in\kk[\Rep(A,V)]^{\chi^n}$ such that $x\in (\Rep(A,V))_f$ and, in addition, we have:
    \begin{enumerate}
    \item [(a)]
    The action map $\GL(V)\times(\Rep(A,V))_f\to(\Rep(A,V))_f$ is a closed morphism, and
    \item [(b)]
    the isotropy group of the point $x$ is finite.
    \end{enumerate}
 \end{enumerate}
\end{definition}
We write $\Rep^{ss}(A,V)$ (resp. $\Rep^{s}(A,V)$) for the set of semistable (resp. stable) points. 
Note that $\Rep^{s}(A,V)\subset \Rep^{ss}(A,V)\subset\Rep(A,V)$.
Remarkably, the $\GL(V)$-orbit of a stable point is an orbit of maximal dimension, equal to $\dim\GL(V)$;
moreover, such a stable orbit is closed in $\Rep^{ss}(A,V)$.
Finally, King \cite{Ki94} introduced a different, purely algebraic, notion of stability for representations of algebras, showing that in the case of quiver representations, his definition of stability is actually equivalent to Mumford' s.\\




Geometric Invariant Theory may be applied to representation spaces, giving a close link between the representation theory of $A$ and the geometric properties of $\Rep(A,V)$; for instance, there exists a one-to-one correspondence between closed orbits of the $\GL(V)$-action on $\Rep( A,V)$ and semisimple\footnote{A \emph{subrepresentation} \cite{EGHLSVY11} of a representation $V$ is a subspace $U\subset V$ which is invariant under all operators $\rho(a)$, with $a\in A$. A nonzero representation $V$ of $A$ is said to be \emph{irreducible} if its only subrepresentations are 0 and $V$ itself. Then a \emph{semisimple} (or \emph{completely reducible}) representation of $A$ is a direct sum of irreducible representations.} representations of $A$ in $V$.
Also, Artin proved that the orbit of a representation is closed in $\Rep(A,V)$ if and only if the corresponding $A$-module is semisimple.

\section{The Van den Bergh functor}
\label{VdB}
The universal representation $\pi\,\colon A\to\End V\otimes A_V$ given in \eqref{pi} enables us to see $\End V\otimes A_V$ as an $A$-bimodule (or equivalently as a left $A^\e$-module).
As $A_V\in \texttt{CommAlg}_\kk$, the image of $A_V$ under the natural inclusion $A_V\hookrightarrow \End V\otimes A_V$ is contained in the center of this $A$-bimodule.
So, we will see $\End V\otimes A_V$ as an $A^\e$-$A_V$-bimodule. 
If $\texttt{Mod}(A_V)$ denotes the category of (left) $A_V$-modules, following \cite{VdB08a}, Lemma 3.3.1, we define the \emph{Van den Bergh functor} by
\begin{equation}
 (-)_V\,\colon \texttt{Bimod}(A)\lto \texttt{Mod}(A_V)\colon\quad M\longmapsto M\otimes_{A^\e}(\End V\otimes A_V).
\label{VdB-functor}
 \end{equation}

 \begin{remark}
 Whereas the functor \eqref{functor-V} acts on the category $\texttt{Alg}_\kk$, the functor \eqref{VdB-functor} takes $A$-bimodules.
  So, from the context, it should be clear which functor is being used. 
Nevertheless, note that \eqref{VdB-functor} is the main object of study of this note.
 \end{remark}

As in \secref{sub:first-repr}, we can describe explicitly $M_V$ as the $A_V$-module generated by symbols
$\{m_{jl}\mid m\in M,\; 1\leq j,l\leq N\}$ satisfying
\[
 (m+m^\prime)_{jl}=m_{jl}+m^\prime_{jl}, \quad (\lambda m)_{jl}=\lambda m_{jl},\quad (am)_{jl}=\sum^n_{r=1}a_{jr}m_{rl}, \quad (ma)_{jl}=\sum^n_ {r=1}a_{rl}m_{jr},
\]
for all $\lambda\in\kk$, $a\in A$.\\

But a better approach consists of defining a similar functor to \eqref{functor-sqrt} applied to bimodules. 
Let $\pi_{\sqrt[V]{-}}\,\colon A\to\End V\otimes \sqrt[V]{A}$ be the universal algebra homomorphism through the adjunction of Proposition \ref{prop-repr-sqrt}(i).
Then $ \sqrt[V]{A}\otimes V$ carries a structure of a left $\End V\otimes \sqrt[V]{A}$-module and right $\sqrt[V]{A}$-module. 
Hence, $\sqrt[V]{A}\otimes V$ is an $A-\sqrt[V]{A}$-bimodule, by restricting the left action via the universal representation $\pi$.
Similarly, we can see $V^*\otimes\sqrt[V]{A}$ as $\sqrt[V]{A}-A$-bimodule.
Then, we can define the functor
\begin{equation}
 \sqrt[V]{-}\,\colon \texttt{Bimod}(A)\lto \texttt{Bimod}(\sqrt[V]{A})\colon \quad M\longmapsto (V^*\otimes\sqrt[V]{a})\otimes_A M\otimes_A(\sqrt[V]{A}\otimes V).
 \label{sqrt-bimodules}
\end{equation}
Using the natural projection $\sqrt[V]{A}\to A_V$, we can consider $A_V$ as a bimodule over $\sqrt[V]{A}$. So,
\begin{equation}
(-)_{\ab}\,\colon\texttt{Bimod}(\sqrt[V]{A})\lto \texttt{Mod}(A_V)\colon\quad M\longmapsto M_{\ab}:=M\otimes_{(\sqrt[V]{A})^\e} A_V.
 \label{functor-ab-bimod}
\end{equation}
Combining \eqref{sqrt-bimodules} and \eqref{functor-ab-bimod}, we obtain an alternative description of the Van den Bergh functor \eqref{VdB-functor}:
\[
(-)_V\,\colon\texttt{Bimod}(A)\longmapsto \texttt{Mod}(A_V)\,\colon\quad M\longmapsto M_V=(\sqrt[V]{M})_{\ab}=M\otimes_{A^{\e}}(\End V\otimes A_V)
\]

The following result is similar to Propositions \ref{prop-repr-sqrt} and \ref{prop-adjunction-repr-sch}:

\begin{proposition}[\cite{BKR13}, Lemma 5.1]
 For any $M\in\texttt{Bimod}(A)$, $N\in\texttt{Bimod}(\sqrt[V]{A})$ and $L\in\texttt{Mod}(A_V)$, we have \begin{enumerate}
  \item [\textup{(i)}]
   There exist a canonical isomorphism
  \begin{equation}
  \Hom_{(\sqrt[V]{A})^\e}(\sqrt[V]{M},N)\simeq\Hom_{A^\e}(M,\End V\otimes N);
  \end{equation}
    \item [\textup{(ii)}]
$M_V$ represents the functor $(-)_V$ in \eqref{VdB-functor}. In other words, 
    \begin{equation}
    \Hom_{A_V}(M_V,L)\simeq\Hom_{A^\e}(M,\End V\otimes L).
    \label{adjunction-bimod}
    \end{equation}
 \end{enumerate}
   \label{lemma5.1.1}
\end{proposition}


Taking $L=M_V$ in \eqref{adjunction-bimod}, we define the universal homomorphism
\[
 \bar{\pi}\,\colon M\lto\End V\otimes A_V
\]
as the image of the identity $\Id_{M_V}\,\colon M_V\to M_V$ under the adjunction \eqref{adjunction-bimod}.
This map allows us to apply the functor \eqref{VdB-functor} to morphisms of bimodules.
Given $f\,\colon M_1\to M_2\in\texttt{Bimod}(A)$, we construct the diagram
\[
 \xymatrix{
 M_1\ar[d]^-f\ar[r]^-{\bar{\pi}_1} & \End V\otimes (M_1)_V
 \\
  M_2\ar[r]^-{\bar{\pi}_2} & \End V\otimes (M_2)_V
  }
\]
Applying the universal property of $\bar{\pi}_1$ to the morphism $\bar{\pi}_2\circ f\,\colon M_1\to \End V\otimes (M_2)_V$,
we obtain a unique morphism
\[
 g\colon\, (M_1)_V\lto M_2)_V
\]
such that $\Id_{\End V}\otimes g$ makes the diagram commute. Define $(f)_V=g$.\\

A final remark is that \eqref{VdB-functor} is an additive functor, which sends projective finitely generated bimodules to projective finitely generated modules.

\section{The Kontsevich--Rosenberg principle}
\label{KR}
The paradigm of non-commutative algebraic geometry is the \emph{Kontsevich--Rosenberg principle},  whereby a structure on
an associative algebra has geometric meaning if it induces standard geometric structures
on its representation spaces.\\

Maybe, the most intuitive way of understanding the Kontsevich--Rosenberg principle is by means of the universal representation $\pi$ (see \eqref{pi}).
Let $\Tr\,\colon\End V\to \kk$ be the linear trace map. 
To each element $a\in A$, we associate the function $\pi(a):=\widehat{a}\,\colon \Rep(A,V)\to \End V$,
$\rho\mapsto\widehat{a}(\rho):=\rho(a)$. By definition, the assignment $a\mapsto \widehat{a}$ gives an associative algebra homomorphism.
Then we can define the composite
\begin{equation}
  \xymatrix{
 A\ar[r]^-\pi & \End V\otimes A_V\ar[r]^-{\Tr\otimes\Id} & A_V,
 }
 \label{mecanismo-KR}
\end{equation}
Therefore, for any $a$, we obtain an element $\Tr\widehat{a}$. Hence, we require that $(a)_V=\Tr\widehat{a}$.
 Since $\Tr(\varphi_1\circ\varphi_2-\varphi_2\circ\varphi_1)=0$, for all $\varphi_1,\varphi_2\in\End V$, 
the map \eqref{mecanismo-KR} descends to a map of vector spaces $A/[A,A]\to A_V$, where $[A,A]$ is the vector subspace of commutators.
By the universal property of symmetric algebras, we finally obtain a morphism of algebras
\begin{equation}
 \Sym^\bullet\left(^A/_{[A,A]}\right)\lto A_V.
\label{KR-functions}
\end{equation}
This is the reason that explains why Kontsevich--Rosenberg \cite{KR00} proposed $ \Sym^\bullet\left(^A/_{[A,A]}\right)$ as the ``algebra of non-commutative functions''.\\

In fact, it can be proved (\cite{Kh12}, Proposition 25) that the image of the map \eqref{KR-functions} is contained in $(A_V)^{\GL(V)}$, the subalgebra of invariants of the action of $\GL(V)$ on $\Rep(A,V)$.
Remarkably, decades ago, Procesi \cite{Pr87} was able to prove that the map
\[
  \Sym^\bullet\left(^A/_{[A,A]}\right)\lto (A_V)^{\GL(V)}
\]
is surjective.\\

But this principle has some limitations. 
Firstly, this principle only works well in practice when the algebra $A$ is smooth (in the sense of \secref{smoothness}).
An insightful perspective was introduced by Berest, Ramadoss and his authors (see \cite{BFR14} for an excellent survey) based on the replacement of $\Rep(A,V)$ by $\DRep(A,V)$,
the differential graded scheme obtained by deriving the classical representation functor \eqref{Rep-functor-comm} in the sense of Quillen's homotopical algebra.
Intuitively, $\DRep(A,V)$ may be regarded as a desingularization of the scheme $\Rep(A,V)$.
Also, in the seminal paper \cite{Gi07}, Ginzburg explored the idea of that any Calabi--Yau algebra of dimension 3 ``arising in nature'' is defined as the quotient of $\kk\langle x_1,...,x_d\rangle$, the free algebra on $d$-generators, by the two-sided ideal generated by all $d$ partial derivatives of a cyclic word (called the \emph{potential}).
The problem, as Ginzburg pointed out in \S 2.1, is that for algebras of this form, the scheme of representations has virtual dimension zero.
This was one of the motivations in \cite{GS10} to extend the representation functor to act on wheelalgebras.
The authors were able to prove that the path algebra of a quiver is a wheeled Calabi--Yau structure, and it induces the Calabi--Yau structure on $\Rep( A,V)$.

\subsection{The Kontsevich--Rosenberg principle for non-commutative differential forms}
One of the main characteristics of the Van den Bergh functor $(-)_V$ defined in \eqref{VdB-functor} is that gives rise to a systematic and unified approach to the Kontsevich--Rosenberg principle.
In fact, this note is an attempt to exploit this idea. 
The following result is essential in the theory since it states that the bimodule $\Omega^1_{\nc} A$ satisfies the Kontsevich--Rosenberg principle.

\begin{proposition}
 We have the isomorphism 
 \[
 (\Omega^1_{\nc}A)_V\simeq \Omega^1_{\comm}(A_V),
 \]
where $ \Omega^1_{\comm}(-)$ denotes the usual module of K\"ahler differentials.
\label{prop-1-forms}
\end{proposition}

\begin{proof}(by Yuri Berest)
The abelianization map gives a natural projection
\begin{equation}
\sqrt[V]{A}  \twoheadrightarrow A_V.
\label{ab-map}
\end{equation}
For any $M \in \texttt{Bimod}(A)$ and $L \in \texttt{Mod}(A_V)$, we have
\[
\Hom_{A_V}(M_V, L) = \Hom_{(\sqrt[V]{A})^\e}(\sqrt[V]{M}, L)
\]
where $ L$ is seen as a symmetric $A_V$-bimodule and hence
as a $\sqrt[V]{A}$-bimodule (via \eqref{ab-map}).
Altervatively, we can take $N=L$ in Proposition \ref{lemma5.1.1}. 
In the particular case when $M=\Omega^1_{\nc}A$, it gives
%
%
%
%
%
%
%
%
%
\begin{align*}
\Hom_{A_V}((\Omega^1_{\nc} A)_V, L) &\simeq \Hom_{\sqrt[V]{A}^\e}(\sqrt[V]{\Omega^1_{\nc}(A)}, L)
\\
&\simeq \Hom_{\sqrt[V]{A}^\e}(\Omega^1_{\nc}(\sqrt[V]{A}), L)
\\
&\simeq\Der(\sqrt[V]{A},L)
\\
&\simeq\Der(A_V,L)
\\
&\simeq\Hom_{A_V}(\Omega_{\comm}^1(A_V),L)
\end{align*}
where the isomorphism  $\Der(\sqrt[V]{A},L) \simeq \Der(A_V,L)$ is due to the fact that any derivation $\sqrt[V]{A}\to L$ with values in a symmetric $A_V$-bimodule
vanishes on commutators in $\sqrt[V]{A}$, and hence induces a (unique) derivation $A_V\to L$. 
The result follows by the Yoneda Lemma.
\end{proof}


The following isomorphism is proved by checking that both sides have the same generators and relations:

\begin{lemma}[\cite{VdB08a}, Lemma 3.3.2]
 Let $P$ be an $A$-bimodule. Then we have
 \[
  (T_AP)_V\simeq \Sym^{\bullet}_{A_V}P_V,
 \]
where $(-)_V$ acts on the left-hand side on the algebra $T_AP$ (not on the $A$-bimodule $T_AP$).
\label{KR-tensor-alg}
\end{lemma}

Then, since $\Omega^\bullet_{\nc}A=T_A(\Omega^1_{\nc}A)$, and $\Omega^\bullet_{\comm}(A_V)=\bigwedge_{A_V}(\Omega^1_{\comm}A_V)$, by Proposition \ref{prop-1-forms} and Lemma \ref{KR-tensor-alg}, 
it is immediate the following isomorphism (\cite{VdB08a}, Corollary 3.3.5):
\[
 (\Omega^\bullet_{\nc}A)_V\simeq \Omega^{\bullet}_{\comm} (A_V).
\]

\begin{lemma}
 Let $\beta\in\DDR{2}{A}$ such that $\du\beta=0(\in\DDR{3}{A})$. Then $\Tr(\widehat{\beta})\in\Omega^2_{\comm}(A_V)$ is closed.
\label{KR-DeRham}
 \end{lemma}

\begin{proof}
 Let $\beta=a_0\du a_1\du a_2\in\Omega^2_{\nc}A$. By \eqref{mecanismo-KR}, $(\beta)_V=\Tr(\widehat{\beta})=\Tr(\widehat{a}_0\du\widehat{a}_1\du\widehat{a}_2)\in\Omega^2_{\comm}(A_V)$,
 giving a map $\varphi\,\colon\Omega^2_{\nc}A\to \Omega^2_{\comm}(A_V)$.
 Since the trace is symmetric (i.e. $\Tr(ab)=\Tr(ba)$), it vanishes on the commutator $[\Omega^i_{\nc}A,\Omega^j_{\nc}A]_{i+j=2}$.
 Hence, we have a map $\DDR{2}{A}\to\Omega^2_{\comm}(A_V)$. Finally, this map commutes with the de Rham differentials and the result follows.
\end{proof}

\subsection{The Kontsevich--Rosenberg principle for double derivations}
In our opinion, part of the interest of the Van den Bergh functor is that gives an immediate proof of the Kontsevich--Rosenberg principle for double derivations, the key notion in this approach to non-commutative algebraic geometry:

\begin{proposition}[\cite{VdB08a}, Proposition 3.3.4, \cite{BKR13}, \S 5.5]
 Assume that $A$ is smooth. Then we have the isomorphism
 \[
  (\D A)_V\simeq \Der (A_V).
 \]
\label{KR-double-derivations}
\end{proposition}
\begin{proof}
 Using \eqref{adjunction-bimod}, we have the following series of isomorphisms
 \begin{align*}
\Der(A_V)&\simeq \Hom_{A^{\e}}(\Omega^1_{\comm} A_V, A_V) 
 \\
&\simeq\Hom_{A^{\e}}(\Omega^1_{\nc}A, \End V\otimes A_V)
  \\
  &\simeq\Hom_{A^\e}(\Omega^1_{\nc}A, {}_{A^{\e}}A^\e)\otimes_{A^{\e}}(\End V\otimes A_V)
\\
  &\simeq \D A\otimes_{A^{\e}} (\End V\otimes A_V)
 \\
 &=(\D A)_V.\qedhere
 \end{align*}
\end{proof}

\subsection{The Kontsevich--Rosenberg principle for bi-symplectic forms}
In \cite{CBEG07}, Theorem 6.4.3 (ii), Crawley-Boevey--Etingof--Ginzburg proved that bi-symplectic forms (in the sense of Definition \ref{bi-sympldef}) satisfy the Kontsevich--Rosenberg. Using the features developed so far, especially the Van den Bergh functor introduced in \eqref{VdB-functor}, we can give an alternative proof of this remarkable result:

\begin{theorem}
Bi-symplectic forms satisfy the Kontsevich--Rosenberg principle.
\end{theorem}

\begin{proof}
Let $\omega\in\DDR{2}{A}$ a bi-symplectic form. 
By Lemma \ref{KR-DeRham}, $(\omega)_V$ is a closed 2-form on $A_V$.
To prove the non-degeneracy,
by Propositions \ref{prop-1-forms} and \ref{KR-double-derivations}, we have the map
\[
(\iota(\omega))_V\,\colon \Der(A_V)\lto \Omega^1_{\comm}(A_V),
\]
which is an isomorphism because $\iota(\omega)\,\colon\D A\to \Omega^1_{\nc}A$ is an isomorphism and functors preserve them.
\end{proof}

\end{document}